\documentclass[a4paper,11pt]{amsart}

\usepackage[latin1]{inputenc}
\usepackage[cyr]{aeguill}
\usepackage[english]{babel}
\usepackage{tikz,hyperref}

\usepackage[T1]{fontenc}
\usepackage{amsmath,bbold}
\usepackage{amsfonts}
\usepackage{amssymb}
\usepackage{amsthm}
\usepackage{enumitem}
\usepackage{graphicx,color}
\usepackage{footnote}
\usepackage{newtxtext}

\hypersetup{colorlinks=true,    linkcolor=blue,
citecolor=red,       filecolor=BrickRed,       urlcolor=darkgreen
}

\newcommand{\EE}{\mathbb{E}}
\newcommand{\PP}{\mathbb{P}}
\newcommand{\RR}{\mathbb{R}}
\renewcommand{\geq}{\geqslant}
\renewcommand{\leq}{\leqslant}
\renewcommand{\epsilon}{\varepsilon}

\def \sclr#1#2{\langle #1,#2\rangle}

\newtheorem{thm}{Theorem} 

\newtheorem{lem}[thm]{Lemma}

\title[Survival probabilities in a cone]{The generating function of the survival probabilities in a cone is not rational}

\author[R.~Garbit]{Rodolphe Garbit}
\address{Universit\'e d'Angers, CNRS, Laboratoire Angevin de Recherche en Math\'ematiques, SFR MATHSTIC, 49000 Angers, France}
\email{rodolphe.garbit@univ-angers.fr}
\author[K.~Raschel]{Kilian Raschel}
\address{CNRS and Universit\'e d'Angers, Laboratoire Angevin de Recherche en Math\'ematiques, SFR MATHSTIC, 49000 Angers, France}
\email{raschel@math.cnrs.fr}
\thanks{This project has received funding from the European Research Council (ERC) under the European Union's Horizon 2020 research and innovation programme under the Grant Agreement No.\ 759702.}


\date{\today}
\begin{document}

\begin{abstract}
We look at multidimensional random walks $(S_n)_{n\geq 0}$ in convex cones, and address the question of whether two naturally associated generating functions may define rational functions. The first series is the one of the survival probabilities $\PP(\tau>n)$, where $\tau$ is the first exit time from a given cone; the second series is that of the excursion probabilities $\PP(\tau>n,S_n=y)$. Our motivation to consider this question is twofold: first, it goes along with a global effort of the combinatorial community to classify the algebraic nature of the series counting random walks in cones; second, rationality questions of the generating functions are strongly associated with the asymptotic behaviors of the above probabilities, which have their own interest. Using well-known relations between rationality of a series and possible asymptotics of its coefficients, recent probabilistic estimates immediately imply that the excursion generating function is not rational. Regarding the survival probabilities generating function, we propose a short, elementary and self-contained proof that it cannot be rational neither.
\end{abstract}

\maketitle 

\section{Introduction}

\subsection*{Main result and our approach}
For a $d$-dimensional random walk $(S_n)_{n\geq 0}$ with integrable and independent increments $X_n=S_n-{S_{n-1}}$   having  common distribution $\mu$, we consider the generating function
\begin{equation}
\label{eq:gen_func}
   F(t)=\sum_{n\geq 0} a_n t^n=\sum_{n\geq 0} \PP^x(\tau>n)t^n,
\end{equation}
where $\PP^x$ is a probability distribution under which the random walk starts at $S_0=x$ and $\tau$ denotes the first exit time from a given cone $K$, i.e.,
\begin{equation*}
   \tau=\inf\{n>0 : S_n\notin K\}.
\end{equation*}   
See \eqref{eq:gen_func_ex} for an explicit computation of \eqref{eq:gen_func} in a simple one-dimensional example. Our first main result can be stated as follows:
\begin{thm}
\label{thm:main}
If the drift $m=\EE X_1$ is not interior to the cone $K$, and if four further assumptions (to be introduced in \ref{hyp:A1}--\ref{hyp:A4} below) are satisfied, then the generating function $F(t)$ in \eqref{eq:gen_func} is not a rational function.
\end{thm}
Our result covers the famous case of walks with small steps in the quarter plane (with arbitrary weights on the steps), but is actually much more general.

The non-rationality of the generating function \eqref{eq:gen_func} is based on the fact that the numbers $a_n$ don't have an asymptotic behavior that is compatible with the Taylor coefficients of a rational function. More precisely, we identify in Theorem~\ref{thm:exp_rate_and_zero_term} a rate $\rho\in(0,1]$ such that
\begin{equation}
\label{eq:asymptotic_a_n}
   a_n=\rho^n B_n,
\end{equation}
with $B_n$ satisfying
\begin{enumerate}[label=(\roman{*}),ref=(\roman{*})]
   \item\label{it:est1}$\sqrt[n]{B_n}\to 1$,
   \item\label{it:est2}$B_n\to 0$.
\end{enumerate}
Using then classical analytic combinatorics techniques (see in particular Theorem~\ref{thm:rational_behavior_GF} and Lemma~\ref{lem:rational_behavior}), one will directly deduce that the generating function \eqref{eq:gen_func} cannot be rational. 

In other words, the two probabilistic estimates \ref{it:est1} and \ref{it:est2} are all we need to prove. In this short paper, we aim at providing proofs of these asymptotic behaviors which are self-contained, and as simple and elementary as possible. Item \ref{it:est1} (in particular the value of the rate $\rho$) is already obtained in \cite{GaRa13}, but we shall give here a simplified proof in our simpler setting. In a restrictive particular case, items \ref{it:est1} and \ref{it:est2} are derived in \cite{Du14}.

\subsection*{Drift inside of the cone}
In case of a drift interior to the cone, the probabilistic behavior is rather constrained as we have $\PP^x(\tau>n)\to \PP^x(\tau=\infty)>0$. The positivity of the escape probability is intuitively clear, based on the law of large numbers and the fluctuations of the random walk; see Lemma~\ref{lem:positive_escape_proba} for a precise statement. Equivalently, in the neighborhood of $t=1$,
\begin{equation*}
    F(t)\sim \frac{\PP^x(\tau=\infty)}{1-t},
\end{equation*}
which contains no contradiction with $F$ being a rational function. However, for one-dimensional walks with bounded jumps, it is proved in \cite[Thm~4]{BaFl-02} that $\PP^x(\tau>n)=\PP^x(\tau=\infty)+\frac{c \rho^n}{n^{3/2}}+\cdots$, with $\rho\in(0,1)$, which is not compatible with $F$ being rational. 

One of the simplest examples for which the rationality of $F$ in \eqref{eq:gen_func} was not solved before the present paper is the following: in the quarter plane $K=\mathbb N^2$, take a uniform distribution $\mu$ on 
\begin{equation*}
     \{(1,0),(0,-1),(-1,0),(0,1),(1,1)\}.
\end{equation*}
Is the generating function $F(t)$ indeed non-rational?

Here, we answer this question and, more generally, solve the problem for the orthant $K=[0,\infty)^d$ and any (weighted) small step walk, i.e., random walk with increments $X_k$ that belong to $\{-1, 0, 1\}^d$ almost surely. If $\PP(X_k\in K)=1$, then the random walk is \textit{trapped} forever in $K$ and $a_n=\PP^x(\tau>n)=1$ for all $n$, so that
$F(t)=\frac{1}{1-t}$ is a rational function. Let us say the walk is \textit{not trapped} if $\PP(X_k\notin K)>0$. Our second main result is the following:
\begin{thm}
\label{thm:inner_drift_not_rational}
For all $d$-dimensional weighted small step walks with a drift interior to the orthant $K=[0,\infty)^d$, not trapped and satisfying \ref{hyp:A2},
the generating function $F(t)$ in \eqref{eq:gen_func} is not rational.
\end{thm}

Here again, the non-rationality of the generating $F(t)$ is obtained as a consequence of estimates on $a_n=\PP^x(\tau>n)$. More precisely, in Theorem~\ref{thm:inner_drift}, we prove that
\begin{equation}
\label{eq:two-term_estimate}
   \PP^x(\tau>n)=\PP^x(\tau=\infty)+\Theta(\rho^n B_n),
\end{equation}
where $\rho\in (0,1)$ and $B_n$ satisfies $\sqrt[n]{B_n}\to 1$ and $B_n\to 0$, and the notation $f_n=\Theta(g_n)$ means that there exist constants $0<c<C$ such that $cg_n\leq f_n\leq Cg_n$.

We could have unified the presentation of the interior and non-interior drift case estimates, since $\PP^x(\tau=\infty)=0$ when the drift is not in $K^o$. However, we choose not to do so because the last double-sided estimate \eqref{eq:two-term_estimate} is obtained only in the small step walk setting.
We leave open the general case of this interesting interior drift problem.

\subsection*{Combinatorial motivations}

Up to a scaling of the $t$-variable, our framework is equivalent to a more combinatorial question, related to the enumeration of walks. More precisely, in case $\mu$ is a uniform distribution on a finite set $\mathcal S$ (with cardinality $\vert\mathcal S\vert$), one has
\begin{equation*}
   F(\vert \mathcal S\vert t) = \sum_{n\geq0} q_n t^n,
\end{equation*}
where $q_n$ denotes the number of walks starting from $x$, having length $n$ and staying in the cone $K$. More generally, when $\mu$ is any distribution, the series $F(t)$ counts the numbers of $\mu$-weighted walks of length $n$ staying in the cone $K$. Accordingly, all our results admit direct combinatorial interpretations.

Recently, in the combinatorial literature, the seminal paper \cite{BoMi10} inspired the following question, which has attracted a lot of attention: given an orthant $K=\mathbb N^d=\{0,1,\ldots\}^d$ and a distribution $\mu$ on $\mathbb Z^d$ (a step set in the combinatorial terminology), is the generating function \eqref{eq:gen_func}, or its refined version
\begin{equation}
\label{eq:gen_func_refined}
   F(x_1,\ldots,x_d;t)=\sum_{n\geq 0}\sum_{(n_1,\ldots,n_d)\in\mathbb N^d} \PP^x(\tau>n,S_n=(n_1,\ldots,n_d))x_1^{n_1}\cdots x_d^{n_d} t^n
\end{equation}
a rational function? An algebraic function? A function satisfying a linear (or non-linear) differential equation? A hypertranscendental function, meaning that like Euler's $\Gamma$ function it does not satisfy any differential equation? In this article, we look at a much simpler question, on the possible rationality of the generating function.

Notice the following relation between \eqref{eq:gen_func} and \eqref{eq:gen_func_refined}: $F(1,\ldots,1;t)=F(t)$. On the other hand, $F(0,\ldots,0;t)$ is the generating function of the excursion sequence
\begin{equation*}
   F(0,\ldots,0;t)=\sum_{n\geq 0} \PP^x(\tau>n,S_n=(0,\ldots,0)) t^n,
\end{equation*}
which will be studied (based on earlier literature \cite{DeWa15}) in Section~\ref{sec:excursions}.

\subsection*{Technical assumptions}
In order to present the hypotheses in the statement of our main results, we need to introduce two objects, through which the exponential rate $\rho$ in \eqref{eq:asymptotic_a_n} will be determined:
\begin{itemize}
\item the Laplace transform $L$ of the increment distribution $\mu$:
\begin{equation}
\label{eq:def_Laplace_transform}
     L(t)=\EE \bigl( e^{\sclr{t}{X_k}}\bigr) =\int_{\mathbb R^d}e^{\sclr{t}{y}}\mu(\text{d}y),
\end{equation}
\item the dual cone $K^*$ associated with $K$ (see Figure~\ref{fig:example_dual} for an example of dual cone):
\begin{equation}
\label{eq:def_dual_cone}
     K^*=\{x\in\mathbb R^d: \sclr{x}{y}\geq 0 \mbox{ for all } y\in K\}.
\end{equation}
Obviously, $K^*$ is a closed convex cone.
\end{itemize}

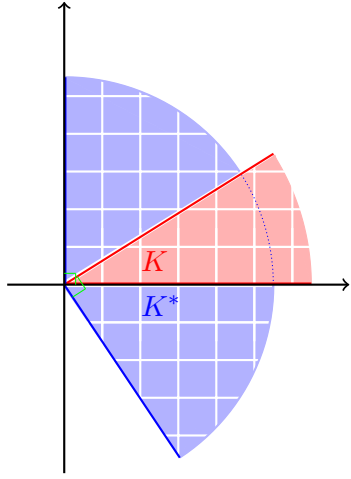
\begin{figure}
\centering
\begin{tikzpicture}
\begin{scope}[scale=0.5]
\draw[red!30,fill=red!30] (0.1,0) -- (6.5,0) -- (5.54,3.44);
\draw[red!30,fill=red!30,domain=0:32.5] plot ({6.5*cos(\x)}, {6.5*sin(\x)});
\draw[blue!30,fill=blue!30,domain=33.5:90] plot ({5.5*cos(\x)}, {5.5*sin(\x)});
\draw[blue!30,fill=blue!30,domain=-56.5:0] plot ({5.5*cos(\x)}, {5.5*sin(\x)});
\draw[blue!30,fill=blue!30] (0,-0.08) -- (5.5,-0.08) -- (3.04,-4.59);
\draw[blue!30,fill=blue!30] (0,0.1) -- (0,5.5) -- (4.65,3);
\draw[white, thick] (-1.5,-1.5) grid (6.5,6.5);
\draw[white, thick] (-1.5,1.5) grid (6.5,-4.5);
\draw[blue,densely dotted,domain=0:32.5] plot ({5.5*cos(\x)}, {5.5*sin(\x)});
\draw[blue!30,thick,domain=33:90] plot ({5.5*cos(\x)}, {5.5*sin(\x)});
\draw[blue!30,thick,domain=-56.5:0] plot ({5.5*cos(\x)}, {5.5*sin(\x)});
\draw[red,thick] (0,0) -- (5.5,3.47);
\draw[red,thick] (0,0.03) -- (6.5,0.03);
\draw[blue,thick] (0,0) -- (3.04,-4.59);
\draw[blue,thick] (0.03,0) -- (0.03,5.5);
\draw[->,thick] (0,-5) -- (0,7.5);
\draw[->,thick] (-1.5,0) -- (7.5,0);
\draw[green] (0,0.3) -- (0.3,0.3);
\draw[green] (0.3,0) -- (0.3,0.3);
\draw[green] (0.34,0.22) -- (0.56,-0.11);
\draw[green] (0.22,-0.33) -- (0.56,-0.11);
\node[above right] at (1.75,0.1) {\textcolor{red}{$K$}};
\node[above right] at (1.75,-1.1) {\textcolor{blue}{$K^*$}};
\end{scope}
\end{tikzpicture}
\caption{A cone $K$ (in red) and its dual cone cone $K^*$ (in blue)}
\label{fig:example_dual}
\end{figure}

Throughout this paper, we make the following assumptions on the cone $K$ and on the distribution $\mu$ of the random walk increments:
\begin{enumerate}[label=(A\arabic{*}),ref=(A\arabic{*})]
   \item\label{hyp:A1}The cone $K$ is convex, closed, with non-empty interior.
   \item\label{hyp:A2}The random walk is truly $d$-dimensional, i.e., there is no $u\not=0$ such that $\sclr{u}{X_1}=0$ almost surely. Moreover, the random walk started at zero can reach the interior $K^o$ of the cone: there exists $k>0$ such that $\PP^0(\tau> k, S_k \in K^o)>0$.
   \item\label{hyp:A3}The random walk increments are $L^1$. We call $m=\EE X_1=\int y \mu(\text{d}y)$ the drift.
   \item\label{hyp:A4}There exists a point $t_0\in K^*$ and a neighborhood $V$ of $t_0$ such that the Laplace transform $L$ of $\mu$ is finite in $V$ and $t_0$ is a minimum point of $L$ restricted to $K^*\cap V$.
\end{enumerate}
Under these assumptions, we proved in \cite{GaRa13} that the exponential rate $\rho$ of the survival probability is equal to $L(t_0)$, meaning that for all $x\in K$,
\begin{equation*}
   \lim_{n\to \infty} \PP^x(\tau>n)^{1/n}= L(t_0).
\end{equation*}
Furthermore, $L(t_0)<1$ if and only if the drift $m$ does not belong to the closed cone $K$.
Here, we shall prove a little bit more:
\begin{thm}
\label{thm:exp_rate_and_zero_term}
Assume hypotheses \ref{hyp:A1}--\ref{hyp:A4} above. If $m\notin K^o$, then 
\begin{equation*}
   \PP^x(\tau>n)=\rho^n B_n,
\end{equation*}
where $\rho=L(t_0)\in (0,1]$, $\sqrt[n]{B_n}\to 1$ and $B_n\to 0$.
\end{thm}
Regarding the interior drift case, we shall prove the following estimate in the small step setting:
\begin{thm}
\label{thm:inner_drift}
For all $d$-dimensional weighted small step walks with a drift interior to the orthant $K=[0,\infty)^d$, not trapped and satisfying \ref{hyp:A2}, we have for all $x\in \mathbb{N}^d$
\begin{equation*}
   \PP^x(\tau>n)-\PP^x(\tau=\infty)=\Theta(\rho^n B_n),
\end{equation*}
where $\rho\in (0,1)$ and $B_n$ satisfies $\sqrt[n]{B_n}\to 1$ and $B_n\to 0$.
\end{thm}

\subsection*{A one-dimensional example}
Take a simple random walk on $\mathbb Z$ with jump probabilities $q$ to the left ($-1$) and $p=1-q$ to the right ($+1$). In this setting,
\begin{equation}
\label{eq:def_tau_dim1}
   \tau = \inf\{n>0 : S_n<0\}= \inf\{n>0 : S_n=-1\}.
\end{equation}
It is well known that, for any positive starting point $x\in\mathbb N$, the series \eqref{eq:gen_func} equals
\begin{equation}
\label{eq:gen_func_ex}
   F(t) = \frac{1-\phi(t)^{x+1}}{1-t}, \qquad \text{with } \phi(t)=\frac{1-\sqrt{1-4pqt^2}}{2pt}.
\end{equation}
It is clear that the function $F$ is never rational; however, it defines an algebraic function (as usual for one-dimensional random walks, see \cite{BaFl-02}).

In the zero drift case (meaning that $p=q=\frac{1}{2}$), expanding \eqref{eq:gen_func_ex} at $t=1$ and using singularity analysis, one finds
\begin{equation*}
   \PP^x(\tau>n) \sim (x+1)\sqrt{\frac{2}{\pi}} \frac{1}{n} \qquad \text{(in particular $\rho=1$).}
\end{equation*}
If the drift is negative ($q>p$), the function $F$ in \eqref{eq:gen_func_ex} is analytic at $1$ as $\phi(1)=1$, and the singularities $t=\pm\frac{1}{2\sqrt{pq}}$ will both contribute to the asymptotics, which reads
\begin{equation*}
   \PP^x(\tau>n) \sim  (x+1)\Bigl({\frac{q}{p}}\Bigr)^{(x+1)/2}\left(\frac{1}{\frac{1}{2\sqrt{pq}}-1}+\frac{(-1)^{x+n}}{\frac{1}{2\sqrt{pq}}+1}\right)\frac{({2\sqrt{pq}})^n}{\sqrt{2\pi} n^{3/2}}.
\end{equation*}
Finally, when the drift is positive ($p>q$), the probability of survival admits the following two-term asymptotics (observe the similarity with the negative drift situation)
\begin{equation}
\label{eq:asymp_survival_drift>0}
   \PP^x(\tau>n) = \left(1-\Bigl(\frac{q}{p}\Bigr)^{x+1}\right)+  (x+1)\Bigl({\frac{q}{p}}\Bigr)^{(x+1)/2}\left(\frac{1}{\frac{1}{2\sqrt{pq}}-1}+\frac{(-1)^{x+n}}{\frac{1}{2\sqrt{pq}}+1}\right)\frac{({2\sqrt{pq}})^n}{\sqrt{2\pi} n^{3/2}}+\cdots.
\end{equation}
The three asymptotics above are obtained by studying the singularities of the generating function \eqref{eq:gen_func_ex} and by using classical transfer theorems on the coefficients.

\section{Survival probability estimates in the non-interior drift case: proof of Theorem \ref{thm:exp_rate_and_zero_term}}

\subsection{Basics on the Laplace transform}
Let us first recall some basic properties. 
The Laplace transform of a random vector $X=(X^{(1)}, \ldots, X^{(d)})\in\RR^d$ with probability distribution $\mu$ is the function $L$ defined for $t\in\RR^d$ by
\begin{equation*}
     L(t)= \EE \bigl( e^{\sclr{t}{X}}\bigr)=\int_{\RR^d} e^{\sclr{t}{y}} \mu(\text{d}y).
\end{equation*}
It is finite in some neighborhood of the origin if and only if $\EE \bigl( e^{\alpha \Vert X\Vert}\bigr) $ is finite for some $\alpha>0$.
If $L$ is finite in some neighborhood of the origin, say $\overline{B(0,r)}$, then $L$ is infinitely differentiable in $B(0,r)$ and its partial derivatives are given there by
\begin{equation*}
     \frac{\partial L(t)}{\partial t_i}=\EE \bigl( X^{(i)}e^{\sclr{t}{X}}\bigr).
\end{equation*}
Therefore, the expectation $\EE X=(\EE X^{(1)}, \ldots, \EE X^{(d)})$ of $X$ is equal to the gradient of $L$ at the origin $\nabla L(0)$. Notice that $X$ is centered (i.e., $\EE X=0$) if and only if $0$ is a critical point of $L$. Since $L$ is a convex function, this means that $0$ is a minimum point of $L$ in $\overline{B(0,r)}$.

Now suppose that $L$ is finite in some ball $\overline{B(t_0,r)}$ and define a new probability measure $\mu_*$ by
\begin{equation*}
   \mu_{*}(\text{d}y)=\frac{e^{\sclr{t_0}{y}}}{L(t_0)}\mu(\text{d}y).
\end{equation*}
The Laplace transform $L_*$ of $\mu_*$ is linked to that of $\mu$ by the relation $L_*(t)=L(t_0+t)/L(t_0)$,
and therefore $L_*$ is finite in some neighborhood of the origin. As a consequence, applying the results above shows that any random vector $X_*$ with distribution $\mu_*$ satisfies:
\begin{itemize} 
\item $\EE \left( e^{\alpha \Vert X_*\Vert}\right)< \infty$ for some $\alpha>0$;
\item $\EE X_*= \nabla L(t_0) / L(t_0)$.
\end{itemize}

As we shall see later, the relevant value of $L$ for our problem is its minimum on the dual cone $K^*$ defined by~\eqref{eq:def_dual_cone}.

We now investigate further properties of $\EE X_*$ when $t_0$ satisfies the assumption~\ref{hyp:A4}, i.e., $t_0$ is a local minimum point of $L$ restricted to $K^*$. By convexity of $L$, the point $t_0$ is necessarily a global minimum on $K^*$; we don't assume $t_0$ to be a global minimum on $\mathbb{R}^d$.
Define the two sets 
\begin{equation*}
   S=\bigl\{ u\in\mathbb{R}^{d} : \exists \epsilon>0, \forall s\in [-\epsilon, \epsilon], t_0+ s u \in K^* \bigr\}
\end{equation*}
and  
\begin{equation*}
   S^+=\bigl\{ u\in\mathbb{R}^{d} : \exists \epsilon>0, \forall s\in [0, \epsilon], t_0+ s u \in K^* \bigr\}.
\end{equation*}
Of course $S\subset S^+$. Since $K^*$ is a convex cone, the set $S$ contains at least $t_0$, while the set $S^+$ contains at least $K^*$.
Assuming~\ref{hyp:A4}, we observe the following:
\begin{itemize}
\item if $u$ belongs to $S^+$, then the function $\phi(s)=L(t_0+su)$ defined on some small interval $[0, \epsilon]$ reaches a minimum at $s=0$, hence 
$\phi'(0)=\sclr{\nabla L(t_0)}{u}\geq 0$. Since $K^*\subset S$, the gradient $\nabla L(t_0)$ belongs to the dual cone $(K^*)^*$ associated with $K^*$;
\item if $u$ belongs to $S$, the function $\phi(s)$ defined on some small interval $[-\epsilon, \epsilon]$ reaches its minimum at $s=0$, hence 
$\phi'(0)=0$. Therefore $\nabla L(t_0)$ is orthogonal to $S$ (and so at least to $t_0$ itself).
\end{itemize} 
Translating these observations in terms of the expectation of $X_*$, we obtain: 

\begin{lem}
\label{lem:geometry_of_new_drift}
Assume \ref{hyp:A1} and \ref{hyp:A4}. The expectation $\EE X_*$ of any random vector with distribution $\mu^*$ belongs to the cone $K$ and is orthogonal to $t_0$.
\end{lem}
\begin{proof}
Since $K$ is a closed convex cone, it is well known that $(K^*)^*=K$ (see Exercise 2.31 in \cite{BoVa04} for example). Everything now follows from the relation $\EE X_*=\nabla L(t_0)/L(t_0)$.
\end{proof}

\subsection{Proof of Theorem~\ref{thm:exp_rate_and_zero_term}}
\label{sec:proof_main_thm}
We shall use the preceding $t_0$ and $\mu_*$ in order to perform an exponential change of measure. For any non-negative and measurable function $f:\RR^n\to[0,\infty)$, elementary algebraic manipulations give:
\begin{align*}
\EE^x\bigl(f(S_1,S_2,\ldots, S_n)\bigr) & = \int_{\RR^n} f\left(x+x_1,x+\sum_{i=1}^2 x_i, \ldots, x+\sum_{i=1}^n x_i\right) \prod_{i=1}^n \mu(\text{d}x_i)\\
	 & = \rho^n \int_{\RR^n} f\left(x+x_1,x+\sum_{i=1}^2 x_i, \ldots, x+ \sum_{i=1}^n x_i\right) e^{-\sclr{t_0}{\sum_{i=1}^n x_i}} \prod_{i=1}^n \mu_*(\text{d}x_i)\\
	 & = \rho^n e^{\sclr{t_0}{x}} \EE_{*}^x\left(f(S_1,S_2,\ldots, S_n)e^{-\sclr{t_0}{S_n}}\right),
\end{align*}
where
\begin{itemize}
\item $\rho=L(t_0)$,
\item $\EE^x_{*}$ is the expectation with respect to $\PP_*^x$, a probability distribution under which $(S_n)_{n\geq0}$ is a random walk with increment distribution $\mu_{*}$ and started at $S_0=x$.
\end{itemize}
Taking $f(s_1,\ldots, s_n)=\prod_{i=1}^n 1_K(s_i)$ leads to
\begin{equation}
\label{eq:fundamental_relation}
\PP^x(\tau >n)=\rho^n e^{\sclr{t_0}{x}} \EE_{*}^x\bigl(e^{-\sclr{t_0}{S_n}}, \tau >n\bigr),
\end{equation}
so that Theorem~\ref{thm:exp_rate_and_zero_term} will follow from the two lemmas below:
\begin{lem}
\label{lem:not_exp_fast}
Assume \ref{hyp:A1}--\ref{hyp:A4}. Then, for all $x\in K$, 
\begin{equation*}
   \lim_{n\to\infty} \sqrt[n]{\EE_{*}^x\bigl(e^{-\sclr{t_0}{S_n}}, \tau >n\bigr)}=1.
\end{equation*}
\end{lem}

\begin{lem}
\label{lem:but_to_zero_anyway}
Assume \ref{hyp:A1}--\ref{hyp:A4}. If the drift $m=\EE X_1$ does not belong to $K^o$, then for all $x\in K$,
\begin{equation*}
   \lim_{n\to\infty} \EE_{*}^x\bigl(e^{-\sclr{t_0}{S_n}}, \tau >n\bigr)=0.
\end{equation*}
\end{lem}

Lemma~\ref{lem:not_exp_fast} is fully proved in \cite{GaRa13}. However, to make our paper self-contained, we propose here a short proof of it in a simplified setting: Instead of \ref{hyp:A2} we will work under the following hypothesis: 
\begin{enumerate}[label=(A\arabic{*}'),ref=(A\arabic{*}')]
\setcounter{enumi}{1}
   \item\label{hyp:A2-prime}there exist $k>0$ and $z\in K^o$ such that $\PP(\tau>k, S_k=z)>0$.
\end{enumerate}
In the majority of classical lattice random walks, \ref{hyp:A2-prime} is satisfied, as for instance for all $74$ non-singular small step random walks considered in \cite{BoMi10}.
\begin{proof}[Proof of Lemma~\ref{lem:not_exp_fast}]
First observe that on the event $\{\tau >n\}$, we have $S_n\in K$, hence $\sclr{t_0}{S_n}\geq 0$ since $t_0\in K^*$. As a consequence
$\EE_{*}^x\bigl(e^{-\sclr{t_0}{S_n}}, \tau >n\bigr)\leq \PP_{*}^x(\tau>n)\leq 1$,
and what remains to prove is that
\begin{equation*}
   \liminf_{n\to\infty} \sqrt[n]{\EE_{*}^x\bigl(e^{-\sclr{t_0}{S_n}}, \tau >n\bigr)}\geq 1.
\end{equation*}
By inclusion of events and basic properties of the $n$-th root limit, it suffices to prove the result for $x=0$, in which case we get rid of the $x$ superscript on $\EE_{*}$ and $\PP_{*}$. We compute a lower bound of the expectation as follows:
\begin{equation*}
   \EE_{*}\bigl(e^{-\sclr{t_0}{S_n}}, \tau >n\bigr)\geq e^{-a_n}\PP_{*} \bigl( \vert \sclr{t_0}{S_n}\vert \leq a_n, \tau >n\bigr),
\end{equation*}
with $a_n=n^{3/4}$. The $e^{-a_n}$ term goes to $1$ in the $n$-th root limit, thus we focus on the probability in the right-hand side.

Assuming \ref{hyp:A2-prime}, we can use the first $k\lfloor\sqrt{n}\rfloor$ steps to push the walk $\lfloor\sqrt{n}\rfloor$ times in the direction $z$ without leaving the cone: by inclusion of events and the Markov property, we have
\begin{equation*}
   \PP_{*} \bigl( \vert \sclr{t_0}{S_n}\vert \leq a_n, \tau >n\bigr)\geq \alpha^{b_n} \PP_{*}^{b_n z} \bigl( \vert \sclr{t_0}{S_{n-kb_n}}\vert \leq a_n, \tau >n-kb_n\bigr),
\end{equation*}
where $\alpha=\PP(\tau>k, S_k=z)>0$ and  $b_n=\lfloor\sqrt{n}\rfloor$. Here again, the $\alpha^{b_n}$ term will disappear in the $n$-th root limit, and the $-kb_n$ does not play any significant role in $n-kb_n$, so we are left to consider the probability
\begin{equation*}
   \PP_{*}^{b_n z} \bigl( \vert \sclr{t_0}{S_{n}}\vert \leq a_n, \tau >n\bigr).
\end{equation*}
At this point, we take into account the ``new drift'' $d=\EE_{*}X_1$ of the random walk under $\PP_{*}$ and consider the centered random walk $\widetilde{S}_n=S_n-nd$. Lemma \ref{lem:geometry_of_new_drift} asserts that:
\begin{itemize}
\item $d$ is orthogonal to $t_0$, so that $\sclr{t_0}{S_{n}}=\sclr{t_0}{\widetilde{S}_n}$,
\item $d$ belongs to $K$, hence 
\begin{equation*}
   \{\tau(\widetilde{S}_{\ell})>n\}:=\{\widetilde{S}_1, \widetilde{S}_2, \ldots, \widetilde{S}_n \in K\} \subset \{S_1, S_2, \ldots, S_n \in K\}=\{\tau>n\}.
\end{equation*}
\end{itemize}
Due to these facts, our probability can be bounded from below by
\begin{align*}
\PP_{*}^{b_n z} \bigl( \vert \sclr{t_0}{\widetilde{S}_{n}}\vert \leq a_n, \tau(\widetilde{S}_{\ell})>n\bigr) &= \PP_{*}\bigl( \vert \sclr{t_0}{b_nz+\widetilde{S}_{n}}\vert \leq a_n, \tau(b_nz+\widetilde{S}_{\ell})>n\bigr)\\
	&= \PP_{*}\bigl( \vert \sclr{t_0}{z+\widetilde{S}_{n}b_n^{-1}}\vert \leq a_nb_n^{-1}, \tau(z+\widetilde{S}_{\ell}b_n^{-1})>n\bigr)\\
	& \geq \PP_{*}\bigl(\Vert \widetilde{S}_{\ell}b_n^{-1}\Vert < \epsilon \mbox{ for all } \ell= 1, \ldots, n \bigr),
\end{align*}
where we have used the homogeneity of the cone, namely $K/b_n=K$ on the second line, and then chosen $\epsilon>0$ so that the ball $B(z,\epsilon)\subset K$. Now recall that, under $\PP_{*}$, the increments $X_n$ of the random walk $S_n$ have a distribution $\mu^*$ with some exponential moments, hence the $X_n$'s are in $L^2$, and so do the increments $X_n-d$ of the centered random walk $\widetilde{S}_n$. Therefore, the Functional Central Limit Theorem \cite[Thm~8.2]{Bil99} is in force and, in conjunction with Portmanteau Theorem \cite[Thm~2.1]{Bil99}, we obtain
\begin{equation*}
   \liminf_{n\to\infty} \PP_{*}^{b_n z} \bigl( \vert \sclr{t_0}{\widetilde{S}_{n}}\vert \leq a_n, \tau(\widetilde{S}_{\ell})>n\bigr) \geq \PP_*\bigl(\Vert B_t \Vert <\epsilon \mbox{ for all } t\in [0,1] \bigr)>0,
\end{equation*}
where $(B_t)_{t\in [0,1]}$ is the image of a standard Brownian motion started at $0$ under a (possibly degenerate) linear transformation. This concludes the proof of the lemma.
\end{proof}

\begin{proof}[Proof of Lemma~\ref{lem:but_to_zero_anyway}]
The proof will be done separately, according to whether $t_0$ is zero or not. First assume $t_0\not=0$. On the event $\{\tau >n\}$, for all $k=1,\ldots, n$, we have that $S_k \in K$, hence $R_k=\sclr{t_0}{S_k}\geq 0$ since $t_0\in K^*$. Therefore
\begin{equation*}
   \EE_{*}^x\bigl(e^{-\sclr{t_0}{S_n}}, \tau >n\bigr)\leq \PP_{*}^x\bigl(R_k \geq 0 \mbox{ for all }k=1,\ldots,n\bigr).
\end{equation*}
Now, under $\PP_{*}^x$, the process $R_k=\sclr{t_0}{S_k}$ is a random walk with increments $Y_k=\sclr{t_0}{X_k}$ having mean
$\sclr{t_0}{\EE_{*}X_1}= 0$ (see Lemma \ref{lem:geometry_of_new_drift}). Since the initial distribution $\mu$ is truly $d$-dimensional and
 $\mu_{*}$ is absolutely continuous with respect to $\mu$, the new distribution $\mu_{*}$ is also truly $d$-dimensional. Thus, under $\PP_{*}^x$, the increments $Y_k$ are non-degenerate (i.e., it does not hold that $Y_k=0$ almost surely). It is well known (see \cite[Thm~1 \& 2 of XII,2]{Fel71}) that for such a one-dimensional random walk, almost surely,
 \begin{equation*}
    -\infty=\liminf R_n < \limsup R_n =+\infty.
 \end{equation*} 
Accordingly,
\begin{equation*}
   \lim_{n\to\infty}\PP_{*}^x\bigl( R_k\geq 0 \mbox{ for all }k=1,\ldots,n\bigr)=\PP_{*}^x\bigl( R_k\geq 0 \mbox{ for all }k\geq 1\bigr)=0.
\end{equation*}

We now turn to the case $t_0=0$. This time $\sclr{t_0}{S_n}=0$, so we don't learn anything by considering this specific one-dimensional random walk. The idea is to replace $t_0$ with an apropriate $\widetilde{t}_0$ and apply the same argument as before. To do this, observe that we know from Lemma \ref{lem:geometry_of_new_drift} that $\EE_{*} X_1$ belongs to the cone $K$, but when $t_0=0$ the change of measure has no effect: $\mu_{*}=\mu$. Hence the original drift $m=\EE X_1$ belongs to~$K$. Since we assumed $m\notin K^o$, we are left with a drift $m$ on the boundary $\partial K$ of the cone $K$.

If $C$ is a closed cone, the interior of its dual cone has the following description:
\begin{equation*}
   (C^*)^o=\bigl\{ x \in \mathbb{R}^d : \sclr{x}{y}>0 \mbox{ for all }y\in C\setminus\{0\} \bigr\}
\end{equation*}
(see Exercise 2.31(d) in \cite{BoVa04} for example). As a consequence, the boundary is given by
\begin{equation*}
   \partial C^*=\bigl\{ x \in C^* :  \sclr{x}{y}=0 \mbox{ for some }y\in C\setminus\{0\} \bigr\},
\end{equation*}
and applying this to the closed convex cone $C=K^*$ gives
\begin{equation*}
   \partial K=\bigl\{ x \in K :  \sclr{x}{y}=0 \mbox{ for some }y\in K^*\setminus\{0\} \bigr\},
\end{equation*}
since $(K^*)^*=K$. Going back to our drift $m\in \partial K$, there exists some $\widetilde{t}_0\in K^*\setminus\{0\}$ such that $\sclr{\widetilde{t}_0}{m}=0$. Setting $\widetilde{R}_k=\sclr{\widetilde{t}_0}{S_k}$, we obtain a centered and non-degenerate one-dimensional random walk such that $S_k\in K$ implies $\widetilde{R}_k\geq 0$. Therefore
\begin{equation*}
   \EE_{*}^x\bigl(e^{-\sclr{t_0}{S_n}}, \tau >n\bigr)=\PP^x(\tau >n) \leq \PP_{*}^x\bigl(\widetilde{R}_k \geq 0 \mbox{ for all }k=1,\ldots,n\bigr),
\end{equation*}
and the conclusion follows as in the first case.
\end{proof}

The proof of Theorem~\ref{thm:exp_rate_and_zero_term} is complete.

\section{Survival probability estimates in the interior drift case: proof of Theorem~\ref{thm:inner_drift}}

In this section, we restrict our attention to the cone $K=[0,\infty)^d$ and small step walks, i.e., random walks on $\mathbb{Z}^d$ with increments $X_k$ satisfying $X_k\in \{-1,0,1\}^d$ almost surely. For such walks, we investigate the case of a drift $m=\EE X_k$ interior to the cone $K$, i.e., such that $\sclr{m}{e_i}>0$ for $i=1,\ldots, d$, where $(e_1, \ldots, e_d)$ denotes the standard basis of $\mathbb{R}^d$. We will use the notation $X_k^{(i)}=\sclr{X_k}{e_i}$.
Since the drift is in the interior of $K$, we know that
\begin{equation*}
   \lim_{n\to\infty}\PP^x(\tau>n)=\PP^x(\tau=\infty)>0
\end{equation*}
for all $x\in K$; see Lemma~\ref{lem:positive_escape_proba} for a precise statement and a proof.

Here we wish to estimate the error term $\delta_n=\PP^x(\tau>n)-\PP^x(\tau=\infty)$. We exclude the case where $\delta_n=0$ for all $n$ by assuming that the random walk is not trapped, i.e., the increments satisfy $\PP(X_k\notin K)>0$.
Under this assumption we will prove Theorem~\ref{thm:inner_drift}, namely that 
\begin{equation*}
   \PP^x(\tau>n)-\PP^x(\tau=\infty)=\Theta\left(\rho^n B_n\right).
\end{equation*}
Before going into the proof, we collect preliminary estimates on $\PP^x(\tau =\infty)$.

\subsection{Exact formula for one-dimensional small step walk}
First of all, we consider the one-dimensional setting with $p=\PP(X_k=1)$, $r=\PP(X_k=0)$, $q=\PP(X_k=-1)$, $p+r+q=1$. Let $\tau$ be as in \eqref{eq:def_tau_dim1} and assume $m=p-q>0$. Then it is known that, for all $x\in\mathbb{N}$, 
\begin{equation*}
   \PP^x(\tau=\infty)=1-\left(\frac{q}{p}\right)^{x+1}.
\end{equation*}
If $q>0$, this can be rewritten as
\begin{equation}\label{eq:rester_pour_toujours}
   \PP^x(\tau=\infty)=1-\gamma e^{-s x},
\end{equation}
where $\gamma=q/p$ and $s>0$ is the unique solution to $e^{-s}=q/p$.

One way to obtain the formula above is to use the discrete harmonicity of the function $u_x= \PP^x(\tau=\infty)$: by the Markov property, we have $u_x=qu_{x-1}+ru_x+pu_{x+1}$
for all $x\geq 1$, which is solved in
$u_x=a+b\left(\frac{q}{p}\right)^x$. Then $a$ and $b$ are determined through initial and limit behaviors of $u_x$.

For future use, we notice the following fact: let 
\begin{equation*}
   L(t)=\EE\bigl(e^{tX_k}\bigr)=p e^t +r +q e^{-t}
\end{equation*}
be the Laplace transform associated with the random walk increments. Its derivative is given by $L'(t)= p e^t -q e^{-t}$. Evaluating at $t=-s$, where $s$ is as above the solution to $e^{-s}=q/p$, leads to
\begin{equation}
\label{eq:inverser_le_drift}
L(-s)=1 \quad \mbox{ and }\quad L'(-s)=q-p=-m<0.
\end{equation}
The last value is exactly the opposite of the drift.

\subsection{Estimate for $\PP^x(\tau<\infty)$ in the $d$-dimensional small step case}
Let us go back to our $d$-dimensional small step walk $(S_n)_n$ with drift $m$ interior to the cone $K=[0,\infty)^d$ and such that 
$\PP(X_k\not\in K)>0$. The simple inclusion of events
\begin{equation*}
   \bigl\{\exists n>0, \sclr{S_n}{e_i} < 0\bigr\} \subset \bigl\{\tau <\infty\bigr\} \subset \cup_{i=1}^d \bigl\{\exists n>0, \sclr{S_n}{e_i}< 0 \bigr\}
\end{equation*}
leads to the bounds
\begin{equation}
\label{eq:encadrement_harmonique}
\frac{g(x)}{d} \leq \PP^x(\tau<\infty) \leq g(x),
\end{equation}
where $g(x)=\sum_{i=1}^d \PP^x(\exists n>0, \sclr{S_n}{e_i}<0)$.
Now, for each $i$, the one-dimensional small step walk $(\sclr{S_n}{e_i})_n$ with increments $X_k^{(i)}$ has a drift
$\EE X_k^{(i)}=\sclr{m}{e_i}>0$. Since $\PP(X_k\not\in K)>0$, the set $I$ of indices $i$ for which $\PP(X_k^{(i)}=-1)>0$ is non-empty, and applying the exact formula \eqref{eq:rester_pour_toujours} of the preceding paragraph, we obtain:
\begin{equation}
\label{eq:somme_uni_dim}
g(x)=\sum_{i\in I} \gamma_i e^{-s_i \sclr{x}{e_i}},
\end{equation}
where
$\gamma_i=\PP( X_k^{(i)}=-1)/\PP(X_k^{(i)}=1)\in (0,1)$ and $s_i>0$ is the unique solution to $e^{-s_i}=\gamma_i$.

\subsection{Proof of Theorem~\ref{thm:inner_drift}}
Fix $x\in \mathbb{N}^d$ and set
\begin{equation*}
   \delta_n=\PP^x(\tau>n)-\PP^x(\tau=\infty)=\PP^x(\tau>n, \mbox{ but }S_m \notin K \mbox{ for some } m>n).
\end{equation*}
By the Markov property of the random walk, we can express $\delta_n$ as follows:
\begin{equation*}
   \delta_n=\EE^x\left(\tau>n, \PP^{S_n}(\tau<\infty)\right),
\end{equation*}
so that inequality \eqref{eq:encadrement_harmonique} leads to
$\delta_n=\Theta(g_n)$, where $g_n=\EE^x(\tau>n, g(S_n))$.
It remains to estimate 
\begin{equation*}
   g_n=\sum_{i\in I} \gamma_i \EE^x\left(\tau>n, e^{ \sclr{S_n}{-s_ie_i} }  \right).
\end{equation*}
To do this, we apply to each term in the sum a specific exponential change of measure. Set
\begin{equation*}
   \mu_{*i}(dy)=\frac{e^{\sclr{-s_ie_i}{y}}}{L(-s_ie_i)} \mu(\text{d}y),
\end{equation*}
where $\mu$ is the common distribution of the increments $X_k$ of the random walk, and $L(t)=\EE(e^{\sclr{t}{X_k}})$ is their Laplace transform. Then basic algebraic manipulations as in Section~\ref{sec:proof_main_thm} lead to
\begin{equation*}
\label{eq:crameratminussi}
   \EE^x\left(\tau>n, e^{ \sclr{S_n}{-s_ie_i} }\right) =L(-s_ie_i)^n e^{\sclr{-s_ie_i}{x}} \PP^x_{*i}\left(\tau>n\right).
\end{equation*}
Now observe that $t\mapsto L(t e_i)=\EE\bigl(e^{t X_k^{(i)}} \bigr)$ is the one-dimensional Laplace transform of the increments $X_k^{(i)}$. Since $s_i$ is the solution to $e^{-s_i}=\gamma_i=\frac{\PP(X^{(i)}_k=-1)}{\PP(X^{(i)}_k=1)}$, we are in the same situation as in \eqref{eq:inverser_le_drift}, so that 
\begin{equation*}
  L(-s_i e_i)=1 \quad \mbox{  and } \quad \frac{\partial L}{\partial t_i}(-s_ie_i)=-\sclr{m}{e_i}<0.
\end{equation*}
Therefore, equation \eqref{eq:crameratminussi} reads
\begin{equation*}
   \EE^x\left(\tau>n, e^{ \sclr{S_n}{-s_ie_i} }\right) = e^{\sclr{-s_ie_i}{x}} \PP^x_{*i}\left(\tau>n\right),
\end{equation*}
and the new drift under $\PP^x_{*i}$, which is given by the gradient of $L$ at the point $-s_i e_i$, has a strictly negative  $i$-th coordinate.
As a consequence, this drift does not belong to the cone $K=[0,\infty)^d$, and it follows from Theorem~\ref{thm:exp_rate_and_zero_term} that
\begin{equation*}
   \PP^x_{*i}(\tau>n)=\rho_i^n B_{i,n},
\end{equation*}
where $\rho_i\in (0,1)$, $\sqrt[n]{B_{i,n}}\to 1$ and $B_{i,n}\to 0$ as $n\to\infty$. Finally, we get
\begin{equation*}
   g_n=\sum_{i\in I} \gamma_i  \rho_i^n B_{i,n},
\end{equation*}
which can be rewritten in the form $g_n=\rho^n B_n$, by selecting 
\begin{equation*}
   \rho=\max\{\rho_i : i \in I\}<1.
\end{equation*}
It is then clear that $\sqrt[n]{B_n}\to 1$ and $B_n\to 0$ and the proof is complete.

\subsection{Positivity of the escape probability}

\begin{lem}
\label{lem:positive_escape_proba}
Assume \ref{hyp:A1} and \ref{hyp:A2}. If the drift $m=\EE X_1$ belongs
to $K^o$, then the function $h(x)=\PP^x(\tau=\infty)$ satisfies:
\begin{enumerate}[label=({\arabic{*}}),ref=({\arabic{*}})]
   \item\label{it1:positive_escape_proba}$h$ is harmonic for the killed random walk, i.e.,
\begin{equation*}
   h(x)=\EE^x(h(S_n), \tau>n ).
\end{equation*}
   \item\label{it2:positive_escape_proba}$h(x)>0$ for all $x\in K$.
   \item\label{it3:positive_escape_proba}$\lim_{t\to\infty}h(tu)=1$ for all $u\in K^o$.
\end{enumerate}
\end{lem}

\begin{proof}
Item \ref{it1:positive_escape_proba} is just the Markov property applied at time $n$. The relation
is valid disregarding the position of the drift.

We now prove \ref{it2:positive_escape_proba}. \textit{First step.} We begin with a simple geometric fact: For any
$z\in K^o$, the non-decreasing sequence of sets $K-kz$ will ultimately
cover the whole space, i.e., $\cup_{k\geq 0} (K-kz)=\RR^d$. To see this,
select $\epsilon>0$ such that
$B(z,\epsilon)\subset K$. For any $x\in \RR^d$, there exists $k>0$ such
that $\Vert x/k\Vert < \epsilon$, hence $z+\frac{x}{k}$ belongs to $K$.
By homogeneity of $K$, it follows that $kz+x\in K$, i.e., $x\in K-kz$.

\textit{Second step.} Let's consider the random walk $(S_n)$ with drift
$m\in K^o$ and select $\epsilon>0$ such that $B(m,\epsilon)\subset K$.
By the strong law of large numbers
$S_n/n\to m$ almost surely, therefore, for almost all $\omega$, there
exists $n_0=n_0(\omega)$ such that
\begin{equation*}
   n\geq n_0 \Rightarrow \Bigl\Vert \frac{S_n(\omega)}{n}- m \Bigr\Vert <\epsilon
\Rightarrow S_n(\omega)\in K.
\end{equation*}
Considering now the first positions $S_1(\omega), S_2(\omega), \ldots,
S_{n_0-1}(\omega)$, the first step of the proof ensures that there
exists $k\geq 0$ such that they all belong to $K-kz$, where $z\in K^0$
is to be fixed in the last step of the proof. Since $K\subset K-kz$
(recall that $K+K\subset K$), all positions $S_n(\omega), n\geq n_0,$ also
belong to $K-kz$ and we obtain the following:
\begin{equation*}
   \PP\bigl(\cup_{k\geq 0}\{S_n \in K-kz \mbox{ for all }n\geq 0\}\bigr)=1.
\end{equation*}
Since the events inside the probability above form a non-decreasing
sequence, it follows that
\begin{equation}
\label{eq:limit_staying_forever}
   \lim_{k\to\infty}\PP\bigl(S_n \in K-kz \mbox{ for all }n\geq 0\bigr)=1.
\end{equation}

\textit{Last step.} To conclude, we invoke hypothesis \ref{hyp:A2} that
claims the existence of an integer $\ell\geq 1$ such that
$\PP(\tau>\ell, S_\ell\in K^o)>0$. Fix some $u\in K^o$. Since
$K^o=\cup_{\lambda>0}(K+\lambda u)$, there is a $z=\lambda u\in K^o$
such that $\PP(\tau>\ell, S_\ell\in K+z)=p>0$. By the Markov property, a
concatenation of $m$ such $\ell$-steps paths leads to
\begin{equation*}
   \PP(\tau>m\ell, S_{m\ell}\in K+mz)\geq p^m>0.
\end{equation*}
On the other hand, it follows from \eqref{eq:limit_staying_forever} that
there exists $k\geq 0$ such that
\begin{equation*}
   \PP(S_n \in K-kz \mbox{ for all }n\geq 0)\geq 1/2.
\end{equation*}
Now choose $m\geq k$. Since $S_{m\ell}\in K+mz$ and $S_n-S_{m\ell}\in K-kz$ imply
$S_n \in K$, we obtain
\begin{equation*}
   \PP(\tau =\infty )\geq \PP(\tau>m\ell, S_{m\ell}\in K+mz)\times
\PP(S_n \in K-kz \mbox{ for all }n\geq 0)>0.
\end{equation*}
We have just proved that $g(0)>0$. The result follows since $g(x)\geq g(0)$ for
all $x\in K$ by inclusion of events.

We conclude with the proof of \ref{it3:positive_escape_proba}. The limit \eqref{eq:limit_staying_forever} obtained in the
\textit{second step} of Item \ref{it2:positive_escape_proba} can be recast as:
\begin{equation*}
   \lim_{k\to\infty}\PP^{kz}(\tau=\infty)=1,
\end{equation*}
where $z$ is any vector in $K^o$. Since $g(x)=\PP^x(\tau=\infty)$ is
non-decreasing in every direction, we are done.
\end{proof}

\section{Classical singularity analysis for rational functions and two elementary lemmas: Proof of Theorems \ref{thm:main} and \ref{thm:inner_drift_not_rational} }

In this section, we show that our estimates on  $a_n=\PP^x(\tau>n)$ given in Theorems \ref{thm:exp_rate_and_zero_term} and~\ref{thm:inner_drift} are not compatible with the generating function $F(t)=\sum_{n\geq 0}a_nt^n$ being rational. 
The starting point is Theorem IV.9 in \cite{FlaSed09} which asserts the following:
\begin{thm}
\label{thm:rational_behavior_GF}
If $F(z)=\sum_{n\geq 0} a_n z^n$ is a rational function that is analytic at $0$ and has poles at points $\alpha_1, \alpha_2, \ldots, \alpha_k$,  then its coefficients are a sum of exponential-polynomials: there exist $k$ polynomials $P_j$ such that, for $n$ larger than some fixed $n_0$,
\begin{equation*}
   a_n=\sum_{j=1}^k P_j(n)\alpha_j^{-n}.
\end{equation*}
\end{thm}

Both estimates in Theorems \ref{thm:exp_rate_and_zero_term} and~\ref{thm:inner_drift} have the following form:
\begin{equation*}
   a_n=a+\Theta(\rho^n B_n),
\end{equation*}
where $a\geq 0$, $\rho\in (0,1]$, $\sqrt[n]{B_n}\to 1$ and $B_n\to 0$. Therefore Theorems~\ref{thm:main} and~\ref{thm:inner_drift_not_rational} asserting the non-rationality of $F$ will follow in both cases from the following elementary lemma.
\begin{lem}
\label{lem:rational_behavior}
Let $c_1, \ldots, c_k$ be distinct non-zero complex numbers and $P_1,\ldots, P_k$ be non-zero complex polynomials. Set $a_n=\sum_{j=1}^k P_j(n) c_j^n$. If $a_n=a+\Theta(\rho^n B_n)$ for some $a\geq 0$, $\rho>0$ and $B_n>0$ such that $\sqrt[n]{B_n}\to 1$, then necessarily $B_n\not\to 0$.
\end{lem}
\begin{proof}
If $a_n=\sum_{j=1}^k P_j(n)c_j^n$, then $a_n-a$ has the same form, thus, without loss of generality, we ca assume $a=0$. Write $c_j=r_j z_j$ with $r_j>0$ and $\vert z_j\vert=1$. Let $r=\max\{r_j : j=1,\ldots, k\}$ and let $J$ be the subset of indices $j$ such that $r_j=r$. Then 
\begin{equation*}
   a_n=\sum_{j=1}^k P_j(n) c_j^n= r^n \left(\sum_{j\in J} P_j(n)z_j^n + o(t^n)\right),
\end{equation*}
where $0<t<1$. For future use, note that the numbers $z_j, j \in J$ are all distinct (this is so since we kept at most one $c_j$ in any fixed  ``direction'' $z_j$: the one with maximum modulus).

We first show that $r=\rho$. Since $a_n=\Theta(\rho^nB_n)$ and $\sqrt[n]{B_n}\to 0$, it follows that $a_n/\rho^n$ goes to one in the $n$-th root limit. Thus, for any $\epsilon>0$,
\begin{equation*}
   \bigl( (1-\epsilon)\rho \bigr)^n\leq a_n \leq \bigl( (1+\epsilon)\rho \bigr)^n
\end{equation*}
for $n$ large enough. Therefore
\begin{equation}
\label{eq:encadrement_polynome}
\left( \frac{(1-\epsilon)\rho}{r} \right)^n \leq \left\vert \sum_{j\in J} P_j(n)z_j^n + o(t^n) \right\vert \leq \left( \frac{(1+\epsilon)\rho}{r} \right)^n 
\end{equation}
for $n$ large enough. If $\rho>r$ then we can choose $\epsilon>0$ such that the lower bound is $A^n$ for some $A>1$. But then we would have
\begin{equation*}
   A^n \leq \left\vert \sum_{j\in J} P_j(n)z_j^n + o(t^n)\right\vert\leq \sum_{j\in J}\vert P_j(n)\vert + \vert o(t^n)\vert
\end{equation*}
and this is impossible since $\sum_{j\in J}\vert P_j(n)\vert$ grows polynomially. On the other hand, if $\rho<r$ then we can choose $\epsilon>0$ such that the upper bound in \eqref{eq:encadrement_polynome} is $A^n$ for some $A<1$. This implies that $\sum_{j\in J} P_j(n)z_j^n \to 0$. Dividing this  by $n^p$, where $p$ stands for the maximum degree of polynomials $P_j$, leads to the convergence
\begin{equation*}
   \sum_{j\in J'} a_j z_j^n \to 0,
\end{equation*}
where $J'\subset J$ is a non-empty subset of indices (those $j$ for which $P_j$ has degree $p$) and the $a_j$'s are non-zero complex numbers.
Since the numbers $z_j$ are distinct complex numbers with modulus $1$, this contradicts Lemma~\ref{lem:Xophe} below. 
The assertion $r=\rho$ is now established, hence we have
\begin{equation*}
  \left\vert  \sum_{j\in J} P_j(n)z_j^n + o(t^n) \right\vert=\frac{a_n}{\rho^n}=\Theta(B_n).
\end{equation*}
We've seen just before that this expression cannot go to zero as $n\to\infty$, thus $B_n\not\to 0$.
\end{proof}

\begin{lem}
\label{lem:Xophe}
Let $z_1, \ldots, z_k$ be distinct complex numbers with modulus $\geq 1$. If
\begin{equation*}
   \lim_{n\to\infty}\sum_{j=1}^k a_j z_j^n= 0,
\end{equation*}
then necessarily $a_1=\cdots = a_k=0$.
\end{lem}
\begin{proof}
Denote by $A_n$ the quantity $\sum_{j=1}^k a_j z_j^n$. Clearly, given any complex numbers $\alpha_0,\ldots,\alpha_{k-1}$,
\begin{equation}
\label{eq:Lagrange_interpolation}
   \sum_{i=0}^{k-1}\alpha_i A_{n+i}=\sum_{j=1}^k a_jP(z_j)z_j^n,
\end{equation}
where $P(z)=\sum_{i=0}^{k-1}\alpha_i z^i$. We can choose the polynomial $P$ so as to have $P(z_1)=1$ and all other $P(z_j)=0$. We then take the limit of \eqref{eq:Lagrange_interpolation} as $n\to\infty$, using the assumption of Lemma~\ref{lem:Xophe}. We find that the term $a_1z_1^n$ should go to zero, which implies that $a_1=0$, since $\vert z_1\vert\geq 1$. A similar reasoning gives that all $a_j=0$, and thus Lemma~\ref{lem:Xophe} is proved.
\end{proof}

\section{The excursion generating function}
\label{sec:excursions}

In this section, we look at lattice random walks in convex cones. Besides the generating function of the survival probabilities \eqref{eq:gen_func}, it is natural to ask whether the excursion generating function
\begin{equation}
\label{eq:gen_func_exc}
   E(t)=\sum_{n\geq 0} \PP^x(\tau>n,S_n=y)t^n
\end{equation}
can be rational, for given starting and ending points $x,y\in K$. When the cone $K$ is an orthant $\mathbb N^d$ and $x=y=(0,\ldots,0)$, the function $E(t)$ reduces to the series $F(0,\ldots,0;t)$ of \eqref{eq:gen_func_refined}. In order to state the result of this section, we introduce the following assumption:
\begin{enumerate}[label=(A\arabic{*}'),ref=(A\arabic{*}')]
\setcounter{enumi}{3}
   \item\label{hyp:A4-prime} There exists a point $\widetilde t_0\in \mathbb R^d$ and a neighborhood $V$ of $\widetilde t_0$ such that the Laplace transform $L$ of $\mu$ is finite in $V$ and $\widetilde t_0$ is a minimum point of $L$ restricted to $V$.
\end{enumerate}
Since $L$ is a convex function, the point $\widetilde t_0$ above is necessarily a global minimum. If $\mu$ is truly $d$-dimensional (as assumed in \ref{hyp:A2}), the function $L $ is strictly convex and a necessary and sufficient condition for the existence of a global minimum
is that the support of $\mu$ is not included in any closed half-space.
\begin{thm}
\label{thm:main_2}
For any distribution satisfying to \ref{hyp:A1}--\ref{hyp:A3}, \ref{hyp:A4-prime} and such that the random walk takes its values on a lattice, the generating function $E(t)$ in \eqref{eq:gen_func_exc} is not a rational function.
\end{thm}

Contrary to our elementary and self-contained proof of Theorem~\ref{thm:main}, we don't have any elementary argument to prove Theorem~\ref{thm:main_2}. Instead, we may give a one-line proof based on earlier literature. Indeed, Denisov and Wachtel provide the following estimate in \cite[Eq.~(10)]{DeWa15} (we use the generalization to convex cones as in \cite[Cor.~1.3]{DeWa19}):
\begin{equation*}
   \PP^x(\tau>n,S_n=y) \sim C(x,y) {\widetilde{\rho}}^n n^{-p-d/2},
\end{equation*}
where $\widetilde{\rho}=L(\widetilde t_0)$ with $\widetilde t_0$ as in \ref{hyp:A4-prime}, $d$ is the dimension and $p>0$ is a geometric quantity related to the cone. One immediately concludes because the exponent of $n$ is negative.

\section*{Acknowledgments}
KR would like to thank Cyril Banderier, Mireille Bousquet-M\'elou, Thomas Dreyfus and \'Eric Fusy for preliminary discussions. We thank Vitali Wachtel and Michael Wallner for interesting discussions. We further thank Christophe Devulder for the elegant proof of Lemma~\ref{lem:Xophe}.

 \bibliographystyle{unsrt}

\begin{thebibliography}{}

\end{thebibliography}


\begin{thebibliography}{99}

\bibitem{BaFl-02}
Banderier, C. and Flajolet, P. (2002).
Basic analytic combinatorics of directed lattice paths.
\textit{Theoret. Comput. Sci.} \textbf{281} 37--80


\bibitem{Bil99}
Billingsley, P. (1999)
\newblock{\it Convergence of Probability Measures. Second edition.}
\newblock Wiley, New York


\bibitem{BoMi10}
Bousquet-M\'elou, M. and Mishna, M. (2010).
\newblock Walks with small steps in the quarter plane. 
\newblock{\it Contemp. Math.} {\bf520} 1--39

\bibitem{BoVa04}
Boyd, S. and  Vandenberghe, L. (2004)
\newblock{\it Convex Optimization.}
\newblock Cambridge University Press


\bibitem{DeWa15}
Denisov, D. and Wachtel, V. (2015). 
\newblock Random walks in cones. 
\newblock{\it Ann. Probab.} {\bf43} 992--1044

\bibitem{DeWa19}
Denisov, D. and Wachtel, V. (2019). 
Alternative constructions of a harmonic function for a random walk in a cone.
\textit{Electron. J. Probab.} \textbf{24} Paper No. 92, 26 pp.


\bibitem{Du14}
Duraj, J. (2014).
\newblock Random walks in cones: the case of nonzero drift. 
\newblock{\it Stochastic Process. Appl.} {\bf124} 1503--1518

\bibitem{Fel71}
{Feller, W.} (1971).
{\it An Introduction to Probability Theory and Its Applications, Volume 2. Second edition.}
Wyley, New York

\bibitem{FlaSed09}
{Flajolet, P. and Sedgewick, R.} (2009).
{\it Analytic Combinatorics.}
Cambridge University Press.

\bibitem{GaRa13}
Garbit, R. and Raschel, K. (2016).
\newblock On the exit time from a cone for random walks with drift.
\newblock{\it Rev. Mat. Iberoam.} {\bf 32} 511--532

\end{thebibliography}
 \makeatletter
 \def\@openbib@code{\itemsep=-3pt}
 \makeatother

\end{document}